\documentclass[final,5p,times,twocolumn]{elsarticle}

\usepackage{amssymb}

\usepackage{lineno}

\usepackage{amsmath}
\usepackage{amsfonts}
\usepackage{mathrsfs}
\usepackage{hyperref,comment} 
\usepackage{pifont}
\usepackage{pgf}
\usepackage{pgfplots}
\usepackage{soul}
\usepackage{color}
\usepackage{calrsfs}
\DeclareMathAlphabet{\pazocal}{OMS}{zplm}{m}{n}

\newcommand{\argmin}{{\rm argmin}}

\newcommand{\dist}{\mbox{\rm dist}\,}

\newcommand{\dom}{\mbox{dom}\,}
\newcommand{\RR}{\mathbb{R}}

\newcommand{\NN}{\mathbb{N}}

\newcommand{\ed}{}

\newtheorem{theorem}{Theorem}[section]
\newtheorem{lemma}[theorem]{Lemma}

\newtheorem{definition}{Definition}
\newtheorem{assumption}{Assumption}
\newtheorem{remark}{Remark}

\newenvironment{proof}[1][]{\noindent {\bf Proof. #1\;}}{\hfill $\Box$\\}
\newenvironment{proofName}[1][]{\noindent {\bf Proof #1.}}{\hfill $\Box$\\}

\colorlet{FRAME}{yellow!5!white}

\journal{\hskip-.14in {\Huge\color{white}$\blacksquare$}}

\begin{document}

\begin{frontmatter}



\title{The value function approach to convergence analysis in composite optimization}


\author{Edouard Pauwels}
\address{IRIT-UPS, 118 route de Narbonne, 31062 Toulouse, France.} 
\ead{edouard.pauwels@irit.fr} 

\begin{abstract}
This works aims at understanding further convergence properties of first order local search methods with complex geometries.  We focus on the composite optimization model which unifies within a simple formalism many problems of this type.  We provide a general convergence analysis of the composite Gauss-Newton method  as introduced in  \cite{burke1995gaussnewton} (studied further in \cite{chong2002convergence,chong2007majorizing,lewis2015proximal}) under tameness assumptions (an extension of semi-algebraicity). Tameness is a very general condition satisfied by virtually all problems solved in practice. The analysis is based on recent progresses in understanding convergence properties of sequential convex programming methods through the value function as introduced in \cite{bolte2016Majorization}.
\end{abstract}


\end{frontmatter}


\begin{figure*}[t!]
\centerline{\fcolorbox{black}{FRAME}{ 
\begin{minipage}{12cm}
	\begin{center}{\bf Composite Gauss-Newton}
\end{center}
Choose $x_0\in \ed{D}$, $\mu_0>0$, $\tau>1$  and iterate\\
\begin{equation}
	\label{eq:mainAlgo}
	\begin{array}{ll}
		\text{\bf Step 1.}& \text{Set $\mu_k = \mu_0$ and compute the candidate iterate:}\\
		& 
		\begin{array}{rl}
			\tilde{x}_{k+1} &\leftarrow\;\; \argmin_{y\in D}\;g(F(x_k) + \nabla F(x_k) (y  - x_k)) + \frac{\mu_k}{2} \|y - x_k\|^2\\ 
		\end{array}\\
		\text{\bf Step 2.}& \text{While } g\left(F(\tilde{x}_{k+1})\right) > g(F(x_k) + \nabla F(x_k) (\tilde{x}_{k+1}  - x_k)) + \frac{\mu_k}{2} \|\tilde{x}_{k+1} - x_k\|^2\\
		&\begin{array}{rl}
			\mu_k &\leftarrow\;\;\tau \mu_k\\
			\tilde{x}_{k+1} &\leftarrow\;\;\argmin_{y\in D}\;g(F(x_k) + \nabla F(x_k) (y  - x_k)) + \frac{\mu_k}{2} \|y - x_k\|^2\\ 
		\end{array}\\
		\text{\bf Step 3.}& \text{Update}\\
		&\begin{array}{rl}
			x_{k+1} &\leftarrow\;\; \tilde{x}_{k+1} 
		\end{array}\\
	\end{array}
\end{equation}
\end{minipage}
}
}
\end{figure*}
\section{Introduction}
In composite optimization, convergence of Gauss-Newton methods is a question that has attracted a lot of research efforts in the past decades. Let us mention a few milestones: criticality of accumulation points was proved in \cite{burke1985descent}, convergence under sharpness assumption around accumulation points is given in \cite{burke1995gaussnewton}, and extensions to weaker regularity conditions are described in \cite{chong2002convergence,chong2007majorizing}. Assymptotic behaviour under prox-regularity and identification under partial smoothness is investigated in \cite{lewis2015proximal}. These results attest to the difficulty of this undertaking. Although the composite model is strongly structured and Gauss-Newton method is explicitly designed to take advantage of it, convergence of iterates always rely on strong local growth conditions around accumulation points. These are often difficult to check in advance for general problems due to the complexity of the optimization model. To our knowledge, a simple and flexible global convergence analysis is still lacking for these methods.


Departing from existing approaches to adress such complex geometries, we rely on tameness assumptions. In the nonsmooth  nonconvex world, this assumption allows to use a powerful geometric property, the so-called nonsmooth Kurdyka-\L ojasiewicz (KL) inequality, which holds true for many classes of functions \cite{loja1963propriete,kurdyka1998gradients,bolte2007lojasiewicz,bolte2007clarke}. We require problem data to be definable, a generalization of the property of being semi-algebraic \cite{DriesMiller98,Coste99}. This rules out non favorable pathological situations such as  wild oscillations (e.g. fractals). This framework is general enough to model the vast majority of functions that can be handled numerically with a classical computer, while providing a sufficient condition for KL inequality to hold \cite{bolte2007clarke}. For a smoother understanding, the reader non familiar with tame geometry may replace ``definable'' by ``semi-algebraic''. Recall that an object is said to be real semi-algebraic if it can be defined as ``the solution set of one of several systems of polynomial equalities and inequalities''.

The use of KL inequality in nonconvex optimization provided significant advances in understanding convergence of first order methods \cite{absil2005convergence,attouch2009convergence,attouch2010proximal,attouch2013convergence,bolte2007lojasiewicz,bolte2013proximal}. However, the application of these techniques in complex geometric settings, such as composite optimization, remains an important challenge. A recent breakthrough has been made in \cite{bolte2016Majorization}, which describes a general convergence analysis of Sequential Quadratic Programming methods \cite{fletcher85,auslender2013extended,snopt}. This is an important example of complex geometric structures with challenging convergence analysis. To overcome the difficulty of dealing with problems with complex geometries in this context, \cite{bolte2016Majorization} has introduced a new methodology based on the so-called value function.

We propose a general convergence guaranty for a variant of the composite Gauss-Newton method \cite{burke1985descent,burke1995gaussnewton}. The main idea consists in viewing Gauss-Newton method along the lines of \cite{bolte2016Majorization} through the value function approach. An important improvement brought to \cite{bolte2016Majorization} is the integration of a general backtracking search in the analysis. This allows to deal with smooth functions whose gradients are merely {\em locally} Lipschitz continuous. This flexibility is extremely important from a practical point of view and requires non trivial extensions (see \cite{noll2012convergence} for works in this direction). To the best of our knowledge this result is new, it relies on easily verifiable assumptions and it is flexible enough to encompass many problems encountered in practice. In addition, we emphasize that it provides a simple and intuitive way to highlight the potential of the value function approach designed in \cite{bolte2016Majorization}.

In Section \ref{sec:pbSetting}, we describe the problem of interest, the main assumptions and the algorithm. We also state our main convergence result. We introduce notations, important definitions and results from nonsmooth analysis and geometry in Section \ref{sec:preliminary}. The value function and its most important properties are described in Section \ref{sec:valueFunction}. Section \ref{sec:proof} contains the proof of the main result.


\section{Problem setting and main result}
\label{sec:pbSetting}
We consider the composite optimization problem.
\begin{align}
	\label{eq:mainProblem}
	\min_{x \in D \subset \RR^n} g(F(x)),
\end{align}
Our main standing assumption is the following.
\begin{assumption}
	\label{ass:mainAssumption}
	$F\colon \RR^n \to \RR^m$ is $\mathcal{C}^2$ and  $g\colon \RR^m \to \RR$ is convex and finite valued. $D \subset \RR^n$ is convex and closed. $F$, $g$ and $D$ are definable {\ed in the same o-minimal structure on the field of real numbers (fixed throughout the text).}
\end{assumption}
Note that Assumption \ref{ass:mainAssumption} ensures that $g$ is locally Lipschitz continuous \cite[Theorem 10.4]{rockafellar1970convex}. For any $i = 1, 2,\ldots, m$, we use the notation $f_i$ for the $\mathcal{C}^2$ function that corresponds to coordinate $i$ of $F$. 
We denote by $\nabla F(x)$ the Jacobian matrix of $F$ at $x$:
\begin{align*}
	\nabla F(x) = \left[ \frac{\partial f_i}{\partial x_j}(x) \right] \in \RR^{m \times n}.
\end{align*}
We will analyse the numerical scheme (\ref{eq:mainAlgo}) which is a backtracking variant of the composite Gauss-Newton descent method \cite{burke1985descent,burke1995gaussnewton,chong2002convergence,chong2007majorizing,lewis2015proximal}. 

\noindent
{\ed
\begin{remark}
	The dynamical feature of the step-size parameter $\mu_k$ is akin to a backtracking procedure. Indeed, Assumption \ref{ass:mainAssumption} ensures that $F$ is locally smooth and $g$ is locally Lipschitz continuous. However the smoothness and Lipschitz continuity moduli may be unknown and not be valid in a global sense. They have to be estimated in an online fashion to prevent unwanted divergent behaviours.
\end{remark}}
%
%
%
%

\bigskip
The next Lemma shows that the algorithm is well defined {\ed and the sequence of objective values is nonincreasing} (the proof is given in Section \ref{sec:valueFunction}). The next Theorem is our main result and the proof is given in Section \ref{sec:proof}.
\begin{lemma}
	\label{lem:wellDefinedZero}
	For each $k$, the while loop stops after a finite number of iterations and we have 
	\begin{align*}
		g\left(F(x_{k+1})\right) \leq g(F(x_k) + \nabla F(x_k) (x_{k+1}  - x_k)) + \frac{\mu_k}{2} \|x_{k+1} - x_k\|^2,\\
	\end{align*}
	{\ed and $\{g(F(x_k))\}_{k \in \NN}$ is a nonincreasing sequence.}
\end{lemma}
\begin{theorem}
	\label{th:mainTheorem}
	Under Assumption \ref{ass:mainAssumption}, we have the alternatives when $k \to +\infty$.
	\begin{itemize}\itemsep-.2em
		\item $\|x_k\| \to + \infty$.
		\item $x_k$ converges to a critical point of Problem (\ref{eq:mainProblem}), the sequence $\|x_{k+1} - x_k\|$ is summable, $\{\mu_k\}_{k\in \NN}$ is bounded.
	\end{itemize}
\end{theorem}
{\ed
\begin{remark}
	In the alternatives of Theorem \ref{th:mainTheorem}, the unbounded case is due to a lack of coercivity rather than a bad adjustment of the local model through $\mu_k$. Indeed, if we suppose that $x_0$ is chosen such that the set $D\cap \left\{ x \in \RR^n; \; g(F(x)) \leq g(F(x_0)) \right\}$ is compact, Lemma \ref{lem:wellDefinedZero} ensures that the divergent option cannot hold and the sequence converges. This phenomenon was guessed in \cite{attouch2010proximal} and also appeared in \cite{bolte2016Majorization}. Accounting for the dynamical feature of $\mu_k$ in our analysis is a contribution of this work.
\end{remark}
}
\medskip
\section{Notations and preliminary results}
\label{sec:preliminary}
\subsection{Notations}
The symbol $\partial$ refers to the limiting subdifferential. The notion of a critical point is that of a limiting critical point: zero is in the limiting subdifferential, a necessary condition of optimality (nonsmooth Fermat's rule). We refer, for instance, the reader to \cite[Chapter 8]{rockafellar1998variational} for further details on the subject. 


{\ed
	An o-minimal structure on the field of real numbers is a structured collection of definable subsets of finite dimensional Euclidean spaces. It is required to satisfy some of the properties of semi-algebraic sets. Semi-algebraic sets form an o-minimal structure but there are many extensions. An introduction to the subject can be found in \cite{Coste99} and a survey of relevant results is available in \cite{Dries-Miller96}. In Assumption \ref{ass:mainAssumption}, we have fixed an o-minimal structure. Definable sets are subsets of Euclidean spaces which belong to it and a definable function is a function which graph is definable.
}

The normal cone to $D$ at $x \in D$ is denoted by $N_D(x)$ and the indicator function of $D$ is denoted by $i_D$ (whose value  is constantly $0$ on $D$, $+\infty$ otherwise). $\|\cdot\|$ denotes the Euclidean norm (which is semi-algebraic). Being given a function $f\colon \RR^p \to \RR$, real numbers $a$ and $b$, we set $[a < f < b] = \{x \in \RR^n:\; a < f(x) < b\}$.

\subsection{Results from nonsmooth analysis}
The next Lemma provides a formula for the subdifferential of the objective function.
\begin{lemma}
	\label{lem:subdiff}
	The chain rule holds for $g(F(\cdot))$.
	\begin{align*}
		\partial g(F(x)) = \nabla F(x)^T v
	\end{align*}
	where $v \in \partial g$ at $F(x)$. Furthermore $g(F(\cdot))$ is subdifferentially regular.
\end{lemma}
\begin{proof}
	Since $g$ is locally Lipschitz continuous, its horizon subdifferential only contains $0$. Since it is convex, it is subdifferentially regular and the result follows from \cite[Theorem 10.6]{rockafellar1998variational}.
\end{proof}
We consider the function $h:\RR^n \times \RR^n \to \RR$, given by
\begin{align}
	\label{eq:defh}
	h(x, y)= g(F(x) + \nabla F(x) (y  - x)) \text{ for any $x, y \in \RR^n$.}
\end{align}
\begin{lemma}
	\label{lem:approx}
	$h$ satisfies the properties:
	\begin{enumerate}
		\item $h$ is continuous and subdifferentially regular.
		\item $h(x, x) = g(F(x))$ {\ed for any $x\in \RR^n$}.
		\item $\frac{\partial h(x, y)}{\partial y} = \{\nabla F(x)^T v;\; v \in \partial g(F(x) + \nabla F(x)(y - x))\}$ {\ed for any $x, y\in \RR^n$}.
		\item $\frac{\partial h(x, y)}{\partial x} = \{( \sum_{i=1}^m v_i \nabla^2 f_i(x) )  (y - x );\; v = (v_1, v_2, \ldots, v_m)^T \in \partial g(F(x) + \nabla F(x)(y - x))  \}$ {\ed for any $x, y\in \RR^n$}.
		\item $h$ is convex in its second argument.
	\end{enumerate}
\end{lemma}
\begin{proof}
	\begin{itemize}\itemsep-.2em
			\item[1, 3, 4.] Continuity follows from Assumption \ref{ass:mainAssumption}, regularity and subdifferential formulas from the same argument as in Lemma \ref{lem:subdiff}.
			\item[2.] Is by the definition of $h$ in (\ref{eq:defh}).
		\item[5.] $y \to g(F(x) + \nabla F(x) (y  - x))$ is the composition of a convex function and an affine map and hence is convex.
	\end{itemize}
	\vskip -.5cm
\end{proof}
\subsection{Results from geometry}
{\ed The next remark gathers important properties of the class of definable functions}.
{\ed
	\begin{remark}
		\label{rem:representableClose}
		Semi-algebraic functions are definable. Definable functions are closed under addition, multiplication, composition, differentiation, projection and partial minimization.	Detailed proof of these facts may be found in \cite{DriesMiller98, Coste99}. See also \cite[Theorem 2.2]{attouch2013convergence} for a specific example in optimization.
	\end{remark}
}
In the context of dynamical systems, a fundamental question is that of the growth of the subdifferential around critical points. This question has a long history in geometry \cite{loja1963propriete, kurdyka1998gradients, bolte2007lojasiewicz, bolte2007clarke}. In the remainder of this text, KL is a short hand for Kurdyka-\L ojasiewicz. We will use the following definition {\ed from \cite{attouch2010proximal}.}

\begin{definition}[KL function]{\rm
	Let $f$ be a proper lower semi-con\-ti\-nuous function from $\RR^p$ to $(-\infty, +\infty]$.
	\begin{enumerate}
		\item[(i)] $f$ has the {\em Kurdyka-\L{}ojaziewicz (KL) property} at $\bar{x} \in \dom \partial f$, if there exist $\alpha \in (0, +\infty]$, a neighborhood $V$ of $\bar{x}$ and a function $\varphi \colon [0, \alpha] \to \RR$, non-negative, concave and continuous, $\mathcal{C}^1$ on $(0,\alpha)$ with $\varphi'>0$ and $\varphi(0) = 0$ such that, for all $x \in V \cap [f(\bar{x}) < f(x) < \alpha]$.
		\begin{equation}\label{loja}\varphi'(f(x) - f(\bar{x})) \, \dist (0, \partial f(x)) \geq 1\end{equation}
	\item[(ii)] The function $f$ is said to be a {\em KL function} if it has the KL property at each point of $\dom\partial f$.
	\end{enumerate}}
\end{definition}
KL property rules out pathological oscilations around critical points. It turns out that all definable functions, even nonsmooth extended-valued functions, have the KL property.
\begin{theorem}[Theorem 11 \cite{bolte2007clarke}]\label{th:KL}
Let $g$ be a proper lower semi-continuous function from $\RR^p$ to $(-\infty, +\infty]$. If $g$ is definable, then $g$ is a KL function.
\end{theorem}
KL property has been extensively used for convergence analysis for nonconvex dynamics both in continuous and discrete time \cite{loja1963propriete,kurdyka1998gradients,absil2005convergence,bolte2007lojasiewicz,attouch2009convergence,attouch2010proximal,attouch2013convergence,bolte2013proximal,bolte2016Majorization}. We conclude this section with a density result whose proof can be found, for example, in \cite[Chapter 6]{Coste99}.
\begin{lemma}
	\label{lem:diffAlmostEvery}
	Let $f\colon \RR^p \to \RR$ be definable, then $f$ is differentiable almost everywhere.
\end{lemma}
\section{Value function and fundamental properties}
\label{sec:valueFunction}
As in \cite{bolte2016Majorization}, we introduce the {\em iteration mapping}, $p_\mu\colon \RR^n \to D$, such that for any $x \in \RR^n$ and $\mu > 0$,
\begin{align}
	\label{eq:defp}
	p_\mu(x)=\argmin_{y \in D}\;h(x, y) + \frac{\mu}{2}\|x - y\|^2.
\end{align}
Note that, from Lemma \ref{lem:approx}, problem (\ref{eq:defp}) is $\mu$-strongly convex, hence, from closedness of $D$, the minimum is indeed attained. According to this definition, the sequence $x_k$ produced by the composite algorithm satisfies $x_{k+1} = p_{\mu_k}(x_k)$. The next result provides a link between the choice of $\mu$ and Step 2 of the algorithm.
\begin{lemma}
	\label{lem:wellDefined}
	Given a compact set $S \subset \RR^n$, there exists $\bar{\mu} > 0$ such that for any $x \in S$ and any $\mu \geq \bar{\mu}$, we have
	\begin{align*}
		g\left(F(p_\mu(x))\right) \leq g(F(x) + \nabla F(x) (p_\mu(x)  - x)) + \frac{\mu}{2} \|p_\mu(x) - x\|^2\\
	\end{align*}
\end{lemma}
\begin{proof}
	The optimization problem in (\ref{eq:defp}) is strongly convex and its data depends continuously on $x$, hence, for $\mu \geq \mu_0 > 0$ and $x \in S$, $p_\mu(x)$ remains bounded. Let $S_1$ be a compact convex set that contains $S \cup \{p_\mu(x);\; x \in S, \mu \geq \mu_0\}$. From Assumption \ref{ass:mainAssumption}, $\nabla F$ is globally Lipschitz continuous on $S_1$ {\ed which ensures the existence of a positive real $a$ such that $\|F(y) - \nabla F(x) (y - x)\| \leq a \|y - x\|^2$ for all $x,y \in S_1$ (see for example the proof of \cite[Lemma 1.2.3]{nesterov2004introduction})}. 
 {\ed Since $S$ and $S_1$ are compact, the set $S_2=\left\{ F(x);\;x \in S_1 \right\} \cup \left\{ F(x) + \nabla F(x)(y - x);\; x\in S, \, y \in S_1 \right\}$ is compact by continuity of $F$ and $\nabla F$. Hence, $g$ is globally Lipschitz continuous on $S_2$ \cite[Theorem 10.4]{rockafellar1970convex}.} This shows existence of a positive real $b$ such that $|g(F(y)) - g(F(x) + \nabla F(x) (y  - x))| \leq ab \|y - x_k\|^2$ for all $y \in S_1$ and $x \in S$. We can take $\bar{\mu} := \max\{\mu_0, 2ab\}$.
\end{proof}

\begin{proofName}[of Lemma \ref{lem:wellDefinedZero}]
	Let $\bar{\mu}$ be given by Lemma \ref{lem:wellDefined} with $S = \{x_k\}$. Condition of Step 2 is automatically satisfied for any $\mu_k \geq \bar{\mu}$ and the while loop must stop. \ed{The nonincreasing property follows by considering in addition the fact that for $k \in \NN$, $x_k \in D$ and hence $x_k$ is always feasible in the minimization problem of Step 1 with value $g(F(x_k))$.}
\end{proofName}

Lemma \ref{lem:approx} provides differentiation rules that relates the iterates $x_k$ to the subdifferential of $g$. However this result is difficult to use in the analysis. Indeed, according to Lemma \ref{lem:approx}, the optimality condition that defines $p_\mu$ can be written
\begin{align}
	\label{eq:optimConditions}
	- \nabla F(x)^T v - \mu (p_\mu(x) - x) \in N_D(p_\mu(x))
\end{align}
where $v \in \partial g(F(x) + \nabla F(x) (p_\mu(x) - x))$. We have no control on the relation between $v$ and $\partial g$ at $F(x)$ or at $F(p_\mu(x))$, which induces a major difficulty in the interpretation of the algorithm as a gradient or a subgradient method.  This features led the authors in \cite{bolte2016Majorization} to introduce and study the value function which we now consider in the composite case with the additional step size parameter feature.  
For any $\mu > 0$, the value function $V_\mu \colon \RR^n \to \RR$, is such that,
\begin{align}
	\label{eq:defV}
	V_\mu(x) &= \min_{y \in D} h(x, y) + \frac{\mu}{2}\|x - y\|_2^2, \text{ for any } x\in \RR^n.
\end{align}
The value function has the subsequent properties.
\begin{lemma}
	\label{lem:valFun}
	\hfill
	\begin{enumerate}
		\item {\ed For any $x\in \RR^n$, $V_\mu(x) = h(x, p_\mu(x)) + \frac{\mu}{2}\|p_\mu(x) - x\|^2$.}
		\item For any $\mu > 0$, $p_\mu$ and $V_\mu$ are definable and continuous on $\RR^n$.
		\item For any $\mu > 0$, the fixed points of $p_\mu$ are exactly the critical points of Problem (\ref{eq:mainProblem}).
		\item For any $\mu > 0$, $V_\mu(x) \leq g(F(x)) - \frac{\mu}{2}\|p_\mu(x) - x\|^2$ for all $x \in D$. 
		\item For any {\ed bounded nonempty set $C$}, there is a constant $K(C) \geq 0$ such that for all $x \in C$ and any $\mu > 0$,
			\begin{align*}
				\dist (0, \partial V_\mu(x)) \leq (K(C) + \mu) \|x - p_\mu(x)\|
			\end{align*}
	\end{enumerate}
\end{lemma}
\begin{proof} We mostly follow \cite[Section 4.2]{bolte2016Majorization}.
	\begin{enumerate}\itemsep-.2em
		\item {\ed This is a consequence of the definition of $p_\mu$ in \eqref{eq:defp} and the definition of $V_\mu$ in \eqref{eq:defV}.}
		\item Continuity of $p_\mu$ holds because of uniqueness of the minimizer in (\ref{eq:defp}) and continuity of $h$. For any $x, z \in \RR^n$, we have
			\begin{align*}
				h(x, p_\mu(x)) +\frac{\mu}{2}\|p_\mu(x) - x\|^2 \leq h(x, p_\mu(z)) +\frac{\mu}{2}\|p_\mu(z) - x\|^2.
			\end{align*}
			From strong convexity and continuity of $h$, $F$ and $\nabla F$, $p_\mu$ must be bounded on bounded sets. Let $x$ converge to $z$ and take $\bar{p}$ any accumulation point of $p_\mu(x)$. By continuity of $h$, we have
			\begin{align*}
				h(z, \bar{p}) +\frac{\mu}{2}\|\bar{p} - z\|^2 \leq h(z, p_\mu(z)) +\frac{\mu}{2}\|p_\mu(z) - z\|^2.
			\end{align*}
			By strong convexity, we must have $\bar{p} = p_\mu(z)$, hence $p_\mu(x) \to p_\mu(z)$. Continuity of $V_\mu$ follows and definability is a consequence of Remark \ref{rem:representableClose}.	
		\item From (\ref{eq:optimConditions}), if $x$ is a fixed point of $p_\mu$, we have $- \nabla F(x)^T v \in N_D(x)$ where $v \in \partial g(F(x))$. Using Lemma \ref{lem:subdiff}, we see that this is exactly the optimality condition for Problem (\ref{eq:mainProblem}).
		\item From Lemma \ref{lem:approx}, and strong convexity of Problem (\ref{eq:defp}), we have {\ed for any $x \in D$},
			\begin{align*}
				V_\mu(x) \leq h(x, x) - \frac{\mu}{2} \|p_\mu(x) - x\|^2 = g(F(x)) - \frac{\mu}{2} \|p_\mu(x) - x\|^2
			\end{align*}
		\item We introduce a parametrized function, for any $\mu> 0$, $e_\mu\colon \RR^n \times \RR^n \to \bar{\RR}$, for any $x,y \in \RR^n$,
			\begin{align*}
				e_\mu(x,y)= h(x, y) + \frac{\mu}{2}\|x - y\|^2+ i_D(y)
			\end{align*}
			Since $V_\mu \colon \RR^n \to \RR$ is definable, using Lemma \ref{lem:diffAlmostEvery}, it is differentiable almost everywhere. Let $S_\mu$ be the set where $V_\mu$ is differentiable (dense in $\RR^n$). Fix a point $\bar{x} \in S_\mu$. We have, for any $\mu,\delta \in \RR^n$, 
			\begin{align*}
				&e(\bar{x} + \delta, p_\mu(\bar{x}) + \mu)\\
				 \geq\;&h(\bar{x} + \delta, p_\mu(\bar{x} + \delta)) + \frac{\mu}{2}\|\bar{x} + \delta - p_\mu(\bar{x} + \delta)\|^2\\
				=\;&V_\mu(\bar{x} + \delta)=V_\mu(\bar{x}) + \left\langle\nabla V_\mu(\bar{x}), \delta \right\rangle + o(\|\delta\|)\\
				=\;&e( \bar{x} , p_\mu(\bar{x}) ) + \left\langle\nabla V_\mu(\bar{x}), \delta \right\rangle + o(\|\delta\|).
			\end{align*}
			This shows that $(\nabla V_\mu(\bar{x}), 0) \in \hat{\partial} e(\bar{x}, p_\mu(\bar{x}))$ where $\hat{\partial}$ denotes the Fr\'echet sudifferential \cite[Definition 8.3]{rockafellar1998variational}. Hence, from Lemma \ref{lem:approx} and \cite[Corollary 10.11]{rockafellar1998variational}, we have
			\begin{align*}
				\nabla V_\mu(\bar{x}) = \left( \sum_{i=1}^m v_i \nabla^2 f_i(\bar{x}) \right)  ( p_\mu(\bar{x}) - \bar{x}) + \mu(\bar{x} - p_\mu(\bar{x})) 
			\end{align*}
			where $v = (v_1, v_2, \ldots, v_m)^T \in \partial g(F(\bar{x}) + \nabla F(\bar{x})(p_\mu(\bar{x}) - \bar{x}))$. By local Lipschitz continuity of $g$, twice continuous differentiability of $F$ and continuity of $p_\mu$, all the quantities that appear in this formula are locally bounded. Hence, for any neighborhood $V$ of $\bar{x}$ there must exist a constant $K$ such that $\|\nabla V_\mu(x)\| \leq (K + \mu)\|x - p_\mu(x)\|$ for all $x \in V \cap S_\mu$. The result is proved by combining continuity of $p_\mu$, definition of the limiting subdifferential \cite[Definition 8.3]{rockafellar1998variational} and the fact that $S_\mu$ is dense in $\RR^n$.
	\end{enumerate}
	\vskip -.5cm
\end{proof}
\section{Proof of Theorem \ref{th:mainTheorem}}
\label{sec:proof}
We extend the proof of \cite[Proposition 4.12]{bolte2016Majorization} to handle the fact that $\mu_k$ is not constant. We actually show that if $\|x_k\| \not\to +\infty$, $\mu_k$ does not diverge. 
{\ed An important ingredient of the proof is the subsequent inequality which can be obtained by combining Lemma \ref{lem:valFun} and Lemma \ref{lem:wellDefinedZero}.}
\begin{align}
	\label{eq:equation(a)}
	V_{\mu_k}(x_{k}) +  \frac{\mu_k}{2}\|x_{k+1} - x_k\|^2 \leq h(x_k, x_k) = g(F(x_k)) \leq V_{\mu_{k-1}}(x_{k-1}).
\end{align}
{\ed We will also rely on properties of $V_{\mu_k}$ and $p_{\mu_k}$ given in Lemma \ref{lem:valFun} (for a fixed $k \in \NN$) and use them in the spirit of \cite{attouch2013convergence,bolte2013proximal}. Finally, we handle the dynamical behaviour of $\mu_k$, $k\geq 0$, defined in Steps 1 and 2 of the algorithm, thanks to Lemma \ref{lem:wellDefinedZero}.}
 Throughout the proof, we assume that $\|x_k\| \not \to + \infty$ that is $\{x_k\}$ has at least one accumulation point.
\paragraph{Case 1: $x_k$ is stationary.}
Suppose that there exists $k_0\geq 0$ such that $x_{k_0+1} = x_{k_0}$. We have a fixed point of $p_{\mu_{k_0}}$, hence of $p_\mu$ for any $\mu > 0$. This implies that $x_{k_0+l} = x_{k_0}$ for all $l \geq 0$. Thus $x_k$ is stationary, hence converges and the increments are summable. Furthermore, from Lemma \ref{lem:wellDefined}, $\mu_k$ must be bounded. Finally,  according to Lemma \ref{lem:valFun}, we have a critical point of Problem (\ref{eq:mainProblem}).
\paragraph{Case 2: $x_k$ is not stationary.}
We now suppose that $||x_{k+1} - x_k|| > 0$ for all~$k\geq 0$. From \eqref{eq:equation(a)}, we have that both $V_{\mu_k}(x_k)$, and $g(F(x_k))$ are decreasing sequences. Let $\bar{x}$ be an accumulation point of $x_k$. The sequence of values $g(F(x_k))$ cannot  go to $-\infty$ and hence converges to $g(F(\bar{x}))$ by continuity. With no loss of generality, we assume that $g(F(\bar{x}))=0$. From \eqref{eq:equation(a)} again, this implies that $\mu_k\|x_{k+1} - x_k\|^2$ is summable and hence goes to $0$ and that $V_{\mu_k}(x_k)$ also converges from above to $g(F(\bar{x}))$.\\

\noindent{\em Definition of a KL neighborhood.} Fix $\delta_1 > 0$. By Lemma \ref{lem:wellDefined}, there must exist a constant $\bar{\mu} > 0$ such that for any $\mu \geq \bar{\mu}$ and any $x$, with $\|x - \bar{x}\| \leq \delta_1$, it must hold that $g(F(p_\mu(x))) \leq V_\mu(x)$. In other words, for any  $k \in \NN$, $\|x_k - \bar{x}\| \leq \delta_1$ implies that $\mu_0  \leq \mu_k \leq \mu_+ := \max\left\{ \mu_0, \tau \bar{\mu} \right\}$. We define the set ${\ed \Theta} = \{\mu_0 \tau^i;\;i\in \NN\} \cap \{t \in \RR;\; \mu_0 \leq t \leq \mu_+\}$
which is a nonempty finite set and satisfies for all $k \in \NN$
\begin{align}
	\label{eq:neighborBeta}
	\|x_k - \bar{x}\| \leq \delta_1 \Rightarrow \mu_k \in \Theta.
\end{align}
For a fixed $\mu \in \Theta$, combining Lemma \ref{lem:valFun} and Theorem \ref{th:KL}$, {V_{\mu}}$ is a KL function. There exists $\delta_{\mu} > 0$, $\alpha_\mu > 0$ and $\varphi_\mu$ which is positive, concave and continuous on $[0, \alpha_\mu]$ and $\mathcal{C}^1$ on $(0, \alpha_\mu)$ with $\varphi_\mu'>0$ and $\varphi_\mu(0) = 0$, such that
	$$\varphi_\mu'(V_\mu(x))\, \dist(0, \partial V_\mu(x)) \geq 1,$$
	for all $x$ such that $\|x-\bar x\|\leq \delta_\mu$ and $x \in [0 < V_\mu < \alpha_\mu]$. Let us consider the following quantities {\ed (recall that $\Theta$ is finite)},
\begin{align}
	\label{eq:defPhi}
	\delta = \min\left\{ \delta_1, \min_{\mu \in \Theta}\left\{ \delta_\mu \right\} \right\} > 0, \quad \alpha = \min_{\mu \in \Theta}\left\{ \alpha_\mu \right\} > 0,\quad\varphi = \sum_{\mu \in \Theta} \varphi_\mu.
\end{align}
{\ed We deduce from properties of each $\varphi_\mu$ for $\mu \in \Theta$ that} $\varphi$ is positive, concave and continuous on $[0, \alpha]$ and $\mathcal{C}^1$ on $(0, \alpha)$ with $\varphi'>0$ and $\varphi(0) = 0$. For any $\mu \in \Theta$, we have
\begin{align}
	\label{eq:generalKL}
	\varphi'(V_\mu(x))\, \dist(0, \partial V_\mu(x)) \geq \varphi_\mu'(V_\mu(x))\, \dist(0, \partial V_\mu(x)) \geq 1, 
\end{align}
for all $x$ such that $\|x-\bar x\|\leq \delta$ and $x \in [0 < V_\mu < \alpha]$. In view of Lemma \ref{lem:valFun}, set $K_2=K\left(\bar B(\bar x,\delta)\right)$, so that for any $x \in B(\bar x, \delta)$ and any $\mu \in \Theta$,
\begin{align}
	\label{eq:stepLength}
	\dist(0, \partial V_\mu(x)) \leq (K_2 + \mu) \|x - p_\mu(x)\|.
\end{align}

	\noindent
	{\em Estimates within the neighborhood.} Let $r\geq s>1$ be some integers and assume that the points $x_{s-1},x_{s}\ldots, x_{r-1}$ belong to $B(\bar x,\delta)$ with $V_{\mu_{s-1}}(x_{s-1}) < \alpha$. Fix $k \in \{s,\ldots,r \}$, we have
\begin{align*}
	&V_{\mu_k}(x_{k}) \\
	\overset{\eqref{eq:equation(a)}}{\leq}\;&V_{\mu_{k-1}}(x_{k-1}) - \frac{\mu_k}{2}||x_{k+1} - x_k||^2 \\
	=\;&V_{\mu_{k-1}}(x_{k-1}) - \frac{\mu_k}{2} \frac{||x_{k+1} - x_k||^2}{||x_{k} - x_{k-1}||} ||p_{\mu_{k-1}}(x_{k-1}) - x_{k-1}|| \\
	\overset{\eqref{eq:stepLength}}{\leq}\;& V_{\mu_{k-1}}(x_{k-1}) - \frac{\mu_k}{2(K_2 + \mu_k)} \frac{||x_{k+1} - x_k||^2}{||x_{k} - x_{k-1}||} \dist(0, \partial V_{\mu_{k-1}}(x_{k-1}))\\
	\overset{\eqref{eq:neighborBeta}}{\leq}\;&V_{\mu_{k-1}}(x_{k-1}) - \frac{\mu_0}{2(K_2 + \mu_0)} \frac{||x_{k+1} - x_k||^2}{||x_{k} - x_{k-1}||} \dist(0, \partial V_{\mu_{k-1}}(x_{k-1})).
\end{align*}
We use $\varphi$ as defined in {\ed \eqref{eq:defPhi}}. This is possible because $V_{\mu_k}(x_k)$ is decreasing, and $V_{\mu_{s-1}}(x_{s-1}) < \alpha$. Let $K = \frac{\mu_0}{2(K_2 + \mu_0)} > 0$, using the monotonicity, the differentiability and the concavity of $\varphi$ we derive
\begin{align}
	&\nonumber\varphi(V_{\mu_k}(x_k))\\
	\leq\;& \varphi(V_{\mu_{k-1}}(x_{k-1})) \\
	&- \varphi'(V_{\mu_{k-1}}(x_{k-1})) \dist(0, \partial V_{\mu_{k-1}}(x_{k-1}))K \frac{||x_{k+1} - x_k||^2}{||x_{k} - x_{k-1}||} \\
	\label{eq:ineqSum}\overset{\eqref{eq:neighborBeta},\eqref{eq:generalKL}}{\leq}\;&\varphi(V_{\mu_{k-1}}(x_{k-1})) - K \frac{||x_{k+1} - x_k||^2}{||x_{k} - x_{k-1}||}.
\end{align}
It is easy to check that for $a>0$  and $b \in \RR$
\begin{equation}
2(a - b) \geq \frac{a^2- b^2}{a}. \label{eq:ineqab}
\end{equation}
We have therefore, for $k$ in $\{s,\ldots,r \}$,
\begin{align*}
	&||x_{k} - x_{k-1}|| \\
	=\;&\frac{||x_{k} - x_{k-1}||^2}{||x_{k} - x_{k-1}||}\\
	=\;& \frac{||x_{k+1} - x_k||^2}{||x_{k} - x_{k-1}||} + \frac{ ||x_{k} - x_{k-1}||^2 - ||x_{k+1} - x_k||^2}{||x_{k} - x_{k-1}||}\\
	\overset{(\ref{eq:ineqab})}{\leq}\;&\frac{||x_{k+1} - x_k||^2}{||x_{k} - x_{k-1}||} + 2(||x_{k} - x_{k-1}|| - ||x_{k+1} - x_k||)\\
	\overset{(\ref{eq:ineqSum})}{\leq}\;&K^{-1}\left(\varphi\left( V_{\mu_{k-1}}(x_{k-1})\right) - \varphi\left(V_{\mu_k}(x_k)\right)\right)\\
	&+  2(||x_{k} - x_{k-1}|| - ||x_{k+1} - x_k||).
\end{align*}
Hence by summation
\begin{align}\label{cauchy}
	\sum_{k=s}^{r} ||x_{k} - x_{k-1}|| \leq\;&K^{-1}\left(\varphi\left( V_{\mu_{s-1}}(x_{s-1})\right) - \varphi\left(V_{\mu_r}(x_r)\right)\right) \\
	&+  2(||x_{s} - x_{s-1}|| - ||x_{r+1} - x_r||).
\end{align}
{\em The sequence remains in the neighborhood and converges.} Take $N$ sufficiently large so that 
\begin{align}
\|x_N-\bar x\|&\leq \frac{\delta}{4},\label{petit0}\\
K^{-1}\varphi\left( V_{\mu_N} (x_{N})\right) &\leq \frac{\delta}{4},\label{petit}\\
{V_{\mu_{N}}}(x_{N})& < \alpha\label{petit1}\\
\label{up} \|x_{N+1}-x_N \| &  < \frac{\delta}{4}.
\end{align}
One can require \eqref{petit0} together with \eqref{petit}, \eqref{petit1} because $\varphi$ is continuous and $V_{\mu_k}(x_k)\downarrow 0$ and (\ref{up}) because $\|x_{k+1}-x_k \| \to 0$. Let us prove that $x_r \in B(\bar{x}, \delta)$ for $r \geq N + 1$. We proceed by induction on $r$. By \eqref{petit0} and \eqref{up}, $x_{N+1} \in B(\bar{x}, \delta)$  thus the induction assumption is valid for $r=N+1$.  Using \eqref{petit1},  estimation \eqref{cauchy} can be applied with $s= N+1$. Suppose that $r \geq N+1$ and $x_{N}, \ldots, x_{r-1} \in B(\bar x,\delta)$, then we have 
\begin{eqnarray*}
&&\|x_{r}-\bar x\| \\
& \leq & \|x_{r}-x_{N}\| + \|x_N-\bar x\| \\
&\overset{\eqref{petit0}}\leq&\sum_{k=N+1}^{r}\|x_{k}-x_{k-1}\|+\frac{\delta}{4}\\
& \overset{\eqref{cauchy}}\leq& K^{-1}\varphi\left(V_{\mu_N}(x_{N})\right) + 2 ||x_{N+1} - x_{N}||+ \frac{\delta}{4}\\
&\overset{\eqref{petit}, \eqref{up}}<&  \delta.
\end{eqnarray*}
Hence $x_{N}, \ldots, x_{r} \in B(\bar x,\delta)$ and the induction proof is complete. Therefore, $x_r \in B(\bar{x}, \delta)$ for any $r \geq N$ and $\mu_r$ takes value in the finite set $\Theta$ and remains bounded for all $r \geq N$. Using \eqref{cauchy} again, we obtain that the series $\sum \|x_{k+1}-x_k\|$ converges, hence $x_k$ also converges by Cauchy's criterion. Let $x_{\infty}$ be its limit, taking $\mu_{\infty}$ any limiting value of $\mu_r$, it must hold that $x_{\infty}$ is a fixed point of $p_{\mu_{\infty}}$ and by Lemma \ref{lem:valFun} a critical point of Problem (\ref{eq:mainProblem}) and the proof is complete.

\section*{Acknowledgments}
Effort sponsored by the Air Force Office of Scientific Research, Air Force Material Command, USAF, under grant number FA9550-15-1-0500. The author would like to thank J\'er\^ome Bolte for his introduction to the topic, constant support and comments on an early draft of this work, Alfred Auslender and Marc Teboulle for their suggestion to consider the composite model.


\begin{thebibliography}{10}

\bibitem{absil2005convergence}
P.~A. Absil, R.~Mahony, and B.~Andrews,
\newblock \emph{Convergence of the iterates of descent methods for analytic cost functions},
\newblock SIAM Journal on Optimization \textbf{16} (2005), no.~2, 531--547.

\bibitem{attouch2009convergence}
H.~Attouch and J.~Bolte,
\newblock \emph{On the convergence of the proximal algorithm for nonsmooth functions involving analytic features}, 
\newblock Mathematical Programming \textbf{116} (2009), no.~1-2, 5--16.

\bibitem{attouch2010proximal}
H.~Attouch, J.~Bolte, P.~Redont, and A.~Soubeyran, 
\newblock \emph{Proximal alternating minimization and projection methods for nonconvex problems: An approach based on the Kurdyka-{\L}ojasiewicz inequality}, 
\newblock Mathematics of Operations Research \textbf{35} (2010), no.~2, 438--457.

\bibitem{attouch2013convergence}
H.~Attouch, J.~Bolte, and B.~F. Svaiter, 
\newblock \emph{Convergence of descent methods for semi-algebraic and tame problems: proximal algorithms, forward--backward splitting, and regularized Gauss--Seidel methods}, 
\newblock  Mathematical Programming \textbf{137} (2013), no.~1-2, 91--129.

\bibitem{auslender2013extended}
A.~Auslender, 
\newblock \emph{An extended sequential quadratically constrained quadratic programming algorithm for nonlinear, semidefinite, and second-order cone programming}, 
\newblock Journal of Optimization Theory and Applications \textbf{156} (2013), no.~2, 183--212.

\bibitem{bolte2007lojasiewicz}
J.~Bolte, A.~Daniilidis, and A.~S. Lewis, 
\newblock \emph{The \L ojasiewicz inequality for nonsmooth subanalytic functions with applications to subgradient dynamical systems}, 
\newblock SIAM Journal on Optimization \textbf{17} (2007), no.~4, 1205--1223.

\bibitem{bolte2007clarke}
J.~Bolte, A.~Daniilidis, A.~S. Lewis, and M.~Shiota, 
\newblock \emph{Clarke subgradients of stratifiable functions},
\newblock SIAM Journal on Optimization \textbf{18} (2007), no.~2, 556--572.

\bibitem{bolte2016Majorization}
J.~Bolte and E.~Pauwels	
\newblock \emph{Majorization-Minimization Procedures and Convergence of SQP Methods for Semi-Algebraic and Tame Programs},
\newblock Mathematics of Operations Research, 2016, in press.

\bibitem{bolte2013proximal}
J.~Bolte, S.~Sabach, and M.~Teboulle, 
\newblock \emph{Proximal alternating linearized minimization for nonconvex and nonsmooth problems},
\newblock Mathematical Programming \textbf{146} (2013), no.~1-2, 459--494.

\bibitem{burke1985descent}
J.~V.~Burke, 
\newblock \emph{Descent methods for composite nondifferentiable optimization problems}. 
\newblock Mathematical Programming, Series A, \textbf{33} (1985), 260--279.

\bibitem{burke1995gaussnewton}
J.~V.~Burke and M.~C.~Ferris. 
\newblock \emph{A Gauss--Newton method for convex composite optimization}. 
\newblock Mathematical Programming \textbf{71.2} (1995) 179--194.

\bibitem{chong2007majorizing}
L.~Chong, and K.~F.~Ng. 
\newblock \emph{Majorizing functions and convergence of the Gauss-Newton method for convex composite optimization}.
\newblock SIAM Journal on Optimization \textbf{18.2} (2007) 613--642.

\bibitem{chong2002convergence}
L.~Chong, and X.~Wang. 
\newblock \emph{On convergence of the Gauss-Newton method for convex composite optimization}.
\newblock Mathematical programming \textbf{91.2} (2002) 349--356.

\bibitem{combpesq11}
P.~L. Combettes and J.-C. Pesquet, 
\newblock \emph{Proximal splitting methods in signal processing}, 
\newblock Fixed-Point Algorithms for Inverse Problems in Science and Engineering, Springer Optimization and Its Applications, Springer New York, 2011, pp.~185--212.

\bibitem{Coste99}
M.~Coste,
\newblock \emph{An introduction to o-minimal geometry}, 
\newblock RAAG Notes, 81 p., Institut de Recherche Math\'{e}matiques de Rennes, November (1999).

\bibitem{Dries-Miller96}
L.~van~den Dries and C.~Miller,
\newblock \emph{Geometric categories and o-minimal structures}, 
\newblock Duke Mathematical Journal \textbf{84} (1996), no.~2, 497--540.

\bibitem{DriesMiller98}
L.~van~den Dries and C.~Miller,
\newblock \emph{Tame Topology and O-minimal Structures}, 
\newblock London Math. Soc. Lecture Note Series 248, Cambridge University Press, (1998).

\bibitem{fletcher85}
R.~Fletcher,
\newblock \emph{An $\ell^1$ penalty method for nonlinear constraints},
\newblock Numerical optimization, SIAM, 1985, pp.~26--40.

\bibitem{snopt}
P.~E. Gill, W.~Murray, and M.~Saunders,
\newblock \emph{SNOPT: An SQP algorithm for large-scale constrained optimization}, 
\newblock SIAM Review \textbf{47} (2005), no.~1,99--131.


\bibitem{kurdyka1998gradients}
K.~Kurdyka, 
\newblock \emph{On gradients of functions definable in o-minimal structures},
\newblock Annales de l'institut Fourier \textbf{48} (1998), no.~3, 769--783.

\bibitem{lewis2015proximal}
A. S.~Lewis and S. J.~ Wright, 
\newblock \emph{A proximal method for composite minimization}, 
\newblock Mathematical Programming (2015), 1--46.


\bibitem{loja1963propriete}
S.~\L{}ojasiewicz,
\newblock \emph{Une propri\'et\'e topologique des sous-ensembles analytiques r\'eels}, 
\newblock Les \'Equations aux D\'eriv\'ees Partielles, vol. 117, \'Editions du Centre National de la Recherche Scientifique, 1963, pp.~87--89.

\bibitem{nesterov2004introduction}
Y.~Nesterov., 
\newblock Introductory Lectures on Convex Optimization,
\newblock Kluwer, Boston, 2004.

\bibitem{noll2012convergence}
D.~Noll and A.~Rondepierre.
\newblock \emph{Convergence of linesearch and trust-region methods using the Kurdyka-\L ojasiewicz inequality}. 
\newblock Computational and Analytical Mathematics. Springer Proceedings in Mathematics (2012), 593--611.

\bibitem{rockafellar1970convex}
R.~T. Rockafellar, 
\newblock \emph{Convex analysis}, Princeton Mathematical Series, 
\newblock No. 28. Princeton University Press, Princeton, N.J., 1970.

\bibitem{rockafellar1998variational}
R.~T. Rockafellar and R.~Wets, 
\newblock \emph{Variational analysis}, 
\newblock vol. 317, Springer, 1998.

\end{thebibliography}
\end{document}